\documentclass[12pt]{article}

\usepackage{amsmath}
\usepackage{amsthm}					
\usepackage{amssymb}				
\usepackage{bm}						

\usepackage{xcolor}					

\usepackage{enumitem}				

\usepackage[T1]{fontenc}			
 \usepackage[utf8]{inputenc}

\usepackage[backend = biber, style = numeric]{biblatex}

\usepackage{tikz}					
\usetikzlibrary{patterns}			
\usepackage{pgfplots} 				
\theoremstyle{definition}
\newtheorem{defi}{Definition}
\newtheorem{ex}{Example}

\theoremstyle{plain}
\newtheorem{lem}{Lemma}
\newtheorem{thm}{Theorem}
\newtheorem{cor}{Corollary}

\theoremstyle{remark}
\newtheorem{rem}{Remark}

\newcommand{\numberthis}{\addtocounter{equation}{1}\tag{\theequation}}

\newcommand{\A}{\mathcal{A}}
\newcommand{\B}{\mathcal{B}}
\renewcommand{\P}{\mathcal{P}}

\newcommand{\Pn}{\mathcal{P}^{(n)}}

\newcommand{\OP}[1]{\mathcal P_\ast^{#1}}

\newcommand{\M}{\mathcal{M}}

\newcommand{\R}{\mathbb{R}}
\newcommand{\N}{\mathbb{N}}

\newcommand{\PE}{h^{PE}}

\newcommand{\card}[1]{\##1}

\newcommand{\todo}[1]{}

\renewcommand{\epsilon}{\varepsilon}
\bibliography{myBibliography}

\begin{document}

\title{Equality of Kolmogorov-Sinai and permutation entropy for one-dimensional maps consisting of countably many monotone parts}

\author{Tim Gutjahr, Karsten Keller}

\maketitle

\begin{abstract}
In this paper, we show that, under some technical assumptions, the Kolmogorov-Sinai entropy and the permutation entropy are equal for one-dimensional maps if there exists a countable partition of the domain of definition into intervals such that the considered map is monotone on each of those intervals.
This is a generalization of a result by Bandt, Pompe and G. Keller, who showed that the above holds true under the additional assumptions that the number of intervals on which the map is monotone is finite and that the map is continuous on each of those intervals.
\end{abstract}

%
%
%
%
%

\section{Introduction}

Determining the \textit{Kolmogorov-Sinai entropy} (\textit{K-S entropy}) of a dynamical system is a key part in the analysis of a system's complexity. This entropy is a measure for the amount of information that is gained on average by observing the system's dynamics. While the K-S entropy has a precise mathematical definition and interesting properties, its computation can be difficult. Therefore, in 2002, Bandt and Pompe introduced the so called \textit{permutation entropy} as an alternative measure for the complexity of a one-dimensional dynamical system which is easier to evaluate numerically than the K-S entropy \cite{bandt2002_1}. The permutation entropy is a measure for the information contained in the ordinal structure of a dynamical system. One can extend the definition of the permutation entropy to multi-dimensional systems by introducing a number of real-valued random variables as observables that each project the multi-dimensional dynamics into the real numbers in which the ordinal structure is then considered.\\
The relatively easy calculation of the permutation entropy led to many practical applications and gave rise to a number of theoretical questions as well (e.g. \cite{Li2007},\cite{Nicolaou2011},\cite{silva2012}). One of the first questions someone could ask is about the equality of the Kolmogorov-Sinai and the permutation entropy. Bandt, Pompe and G. Keller showed that those entropies are equal for one-dimensional interval maps  if there exists a finite partition of the domain of definition into intervals such that the considered map is monotone and continuous on each of those intervals \cite{bandt2002_2}. It was also shown that the permutation entropy is an upper bound for the K-S entropy for all one-dimensional systems (see \cite{bandt2002_2} or \cite{amigo2005}). This inequality can be generalized to multidimensional maps under sufficiently general conditions \cite{unakafov2012}. It is still an open question if and how the condition of piecewise monotony for the equality of the entropies can be generalized to a larger class of one-dimensional maps.\\
In this paper, we are able to show that the equality of K-S and permutation entropy still holds true if we omit the condition of continuity and if there exists a countable partition of the domain of definition into intervals such that the one-dimensional map is monotone on each of those intervals. Unlike in the paper of Bandt, Pompe and G. Keller \cite{bandt2002_2}, we do not require that this partition into intervals is finite.

\subsection{Basic definitions}
We are interested in the complexity of one-dimensional dynamical systems. In this paper, such systems will be a tupel $(\Omega,\B(\Omega),\mu,T)$ consisting of a compact metric space $\Omega\subseteq \R$, the Borel $\sigma$-algebra $\B(\Omega)$ on $\Omega$, a $\B(\Omega)$-measurable map $T:\Omega\rightarrow\Omega$ and a probability measure $\mu$ defined on $\B(\Omega)$. The system $(\Omega,\B(\Omega),\mu,T)$ is called measure-preserving dynamical system if, additionally, the map $T$ is measure-preserving, i.e. $\mu(T^{-1}(A)) = \mu(A)$ for all $A\in \B(\Omega)$.\\
%
%

The value of the Kolmogorov-Sinai entropy depends on the position of the elements of orbits $(\omega,T(\omega),T^{2}(\omega),\ldots)$ with respect to a finite or countable partition. The iterates $T^t(\omega)$ can be recursively defined as $T^t(\omega) = T(T^{(t-1)}(\omega))$ for $t\in \N$ and $\omega \in \Omega$ with $T^0(\omega) = \omega$.\\
For the definition of the permutation entropy, we investigate for which $s,t \in \N_0$ the inequality $T^s(\omega)\leq T^t(\omega)$ holds true. To simplify our argumentation we want to exclude the possibility that $T^s(\omega)$ is equal to $T^t(\omega)$ for $s,t\in \N$ with $s\neq t$, so that the inequalities $T^s(\omega)\leq T^t(\omega)$ and $T^t(\omega)\leq T^s(\omega)$ are mutually exclusive. To achieve this, we require that $T$ is aperiodic with regard to $\mu$, which means that $\mu(\bigcup_{n=1}^\infty\{\omega\in\Omega|~T^{ n}(\omega)=\omega\}) = 0$ holds true. For aperiodic maps, the probability of two different iterates of a single point being equal is zero. Being aperiodic is not a significant restriction though, as noted in section~\ref{sec:discussion}.\\

When determining the complexity of a dynamical system, we consider the probabilities of $\omega,T(\omega),T^{ 2}(\omega),\ldots,T^{(n-1)}(\omega)$ lying within specific sequences of sets. This leads to the definition of the Kolmogorov-Sinai entropy:

\begin{defi}[Kolmogorov-Sinai entropy]
Let $(\Omega,\A,\mu,T)$ be a measure-preserving dynamical system and $\P=\{P_i\}_{i\in I}$ a partition of $\Omega$ with some finite or countable index set $I$.
Define for $n\in \N$ and a multi index $\bm{i}=(i_0,i_1,\ldots i_{n-1})\in I^n$ the set
\[ P(\bm{i}):= \bigcap_{k=0}^{n-1} T^{-k}(P_{i_k}) =  P_{i_0}\cap T^{-1}(P_{i_1}) \cap\ldots T^{-n+1}(P_{i_{n-1}}) \]
and the partition $\Pn := \{ P(\bm{i})|~\bm{i}\in I^n \}$. The\textit{ Kolmogorov-Sinai entropy} (or \textit{entropy rate}) of $T$ \textit{with regard to the partition} $\P$ is defined as
\begin{equation}
h(T,\P):= \lim_{n\to \infty} \frac{1}{n}H(\Pn),
\label{eq:KS}
\end{equation}
where $H(\Pn)= -\sum_{i\in I^n}\mu(P(\bm i))\log(\mu(P(\bm i)))$ is the \textit{Shannon entropy} of the partition $\Pn$. By
\[ h(T) := \sup_{H(\P)<\infty}h(T,\P) \]
the \textit{Kolmogorov-Sinai entropy} of $T$ is defined, where the supremum is taken over all finite or over all countable partitions with finite entropy.
\end{defi}

\begin{rem}
Originally, the Kolmogorov-Sinai entropy was defined as the supremum of the entropy rates over finite partitions, disregarding countable partitions. However, according to Abramov's Theorem, the supremum of the entropy rates over all countable partitions with finite entropy is not larger than the supremum of the entropy rates over all finite partitions \cite{keller1998}.
\end{rem}

If we do not consider the probabilities of $\omega,T(\omega),T^{ 2}(\omega),\ldots,T^{(n-1)}(\omega)$ lying within specific sequences of sets but instead the probabilities of\\
$\omega,T(\omega),T^{ 2}(\omega),\ldots,T^{(n-1)}(\omega)$ being in some specific order relation, we get the definition of the permutation entropy:

\begin{defi}[Permutation entropy]
Let $(\Omega,\B(\Omega),\mu,T)$ be a measure-\\
preserving dynamical system with $\Omega \subseteq \R$ and
\[ \Pi_n:=\{(\pi_0,\pi_1,\ldots \pi_{n-1}) \in \{0,1,\ldots n-1\}^n|~\pi_i\neq\pi_j~\textrm{for}~i\neq j\} \]
the set of all permutations of length $n$. For a permutation $\pi=(\pi_0,\pi_1,...,\pi_{n-1})\in \Pi_n$ we denote the set of points with ordinal pattern $\pi$ by
\begin{equation}
P_\pi:=\{\omega\in\Omega|~T^{\pi_0}(\omega)\leq T^{\pi_1}(\omega)\leq\ldots\leq T^{\pi_{n-1}}(\omega) \}
\label{eq:encodingPi}
\end{equation}
and by $\OP{n} := \{P_\pi|~ \pi\in \Pi_n\}$ the partition of $\Omega$ into these sets. The permutation entropy of $T$ is defined as
\begin{equation}
\PE(T) := \liminf_{n\to \infty}\frac{1}{n} H(\OP{n}),
\label{eq:PE}
\end{equation}
where $H(\OP{n})= -\sum_{\pi\in \Pi_n}\mu(P_\pi)\log(\mu(P_\pi))$ is the Shannon entropy of the partition $\OP{n}$.
\end{defi}

\paragraph{Remarks}
\begin{enumerate}
\item Like Amig\'{o}, Kennel and Kocarev \cite{amigo2005}, we use the limit inferior for the definition of the permutation entropy in \eqref{eq:PE} instead of the limit like Bandt, Pompe and G. Keller \cite{bandt2002_2}. This is because, unlike in \eqref{eq:KS}, one does not know whether $\frac{1}{n} H(\OP{n})$ converges for $n\to \infty$. Alternatively, one could use the limit superior in \eqref{eq:PE} like, for example, A. M. Unakafov and V. A. Unakafova in \cite{unakafov2012}. By replacing the limit inferior with the limit superior in the argumentation of this paper each statements remains valid so that we can conclude that the limit in \eqref{eq:PE} does exist for the here considered class of maps $T$.
\item Technically speaking, the collection $\OP{n}$ of sets $P_\pi, \pi \in \Pi_n$, is not actually a partition. A point $\omega\in \Omega$ with $T^{i}(\omega) = T^{j}(\omega)$ for some $i,j\in\N_0$ with $0\leq i<j\leq n-1$ belongs to at least two sets $P_\pi \in \OP{n}$. However, such points belong to the set of (pre-)periodic points, which has measure zero for aperiodic maps $T$. So the sets of points $P_\pi$ with ordinal patterns $\pi\in \Pi_n$ are only disjoint $\mu$-almost surely. This is not a problem because the value of the entropy is not affected by sets of measure zero.
\end{enumerate}

\subsection{The Main result}

Since taking the supremum over all finite partitions is necessary to calculate the Kolmogorov-Sinai entropy, its determining can be difficult. There are theoretical results that ensure, under some conditions, the existence of a partition, such that the entropy rate with regards to this partition yields the K-S entropy. However, in practice, one does not know how such partitions look like. Additionally, this partition depends on the dynamics $T$ whose precise description might be unknown as well for practical applications.\\
The permutation entropy has the advantage that it can be calculated without having to find such partitions. The ordinal patterns necessary for the calculating of the permutation entropy automatically partition the space $\Omega$ in a way that can capture the information of a system, independently of the considered map $T$.\\
We would like to know whether the complexity of the ordinal structure is equal to the complexity of partitions generated by iteration of $T$. That is, does
\[ \PE(T) = h(T) \]
hold true? It is possible to show that
\begin{equation}
\PE(T) \geq h(T)
\label{eq:inequality}
\end{equation}
is fulfilled for a measure-preserving dynamical system $(\Omega,\B(\Omega),\mu,T)$ with $\Omega\subseteq\R$ (see \cite{bandt2002_2} or \cite{amigo2005}).

Bandt, Pompe and G. Keller proved the following statement in 2002 \cite{bandt2002_2}:

\begin{thm}
Let $(\Omega,\B(\Omega),\mu,T)$ be a measure-preserving dynamical system for some interval $\Omega\subseteq \R$. If $T$ is piecewise monotone and continuous on each monotony interval of $T$, then
\begin{equation}
\PE(T) = h(T)
\label{eq:PE=KS}
\end{equation}
holds true.
\label{thm:bandtPompe}
\end{thm}

Being piecewise monotone is defined in the following way:

\begin{defi}[Piecewise monotony]
Let $T:\Omega \rightarrow \Omega$ be a map for $\Omega \subseteq \R$. $T$ is called \textit{monotone} on a set $M\subseteq\Omega$ if
\[ \omega_1 \leq \omega_1 \quad\textrm{implies}\quad T(\omega_1)\leq T(\omega_2)\quad\mathrm{for}~\textrm{all}~\omega_1,\omega_2 \in M\]
or
\[ \omega_1 \leq \omega_1 \quad\textrm{implies}\quad T(\omega_1)> T(\omega_2)\quad\mathrm{for}~\textrm{all}~\omega_1,\omega_2 \in M.\]
The map $T$ is called \textit{piecewise monotone} if there exists a finite partition\\
$\M=\{M_1,M_2,...,M_n\}$ of $\Omega$ into intervals, including single point sets, such that $T$ is monotone on each interval $M\in\M$. (Notice that, unlike Bandt, Pompe and G. Keller, we do not require that $T$ is continuous on each interval $M\in \M$.)\\
We call $T$ \textit{countable piecewise monotone} if there exists a countable partition $\M=\{M_1,M_2,...\}$ of $\Omega$ into intervals, including single point sets, such that $T$ is monotone on each interval $M\in\M$.\\
We call $\M$ \textit{partition into monotony intervals} of $T$.
\end{defi}

Given a compact metric space $\Omega\subseteq \R$, the fact that $T:\Omega \rightarrow \Omega$ is (countable) piecewise monotone automatically implies that $\Omega$ can be represented as a finite (or countable) union of intervals. However, $\Omega$ itself can be a more general set than a single interval.\\

It was not known whether \eqref{eq:PE=KS} is true for a more general case than piecewise monotony.

\begin{ex}[Gauss function]
The map $T:[0,1]\rightarrow [0,1]$ with
\[ T(\omega) = \begin{cases}1/\omega~\mathrm{mod}~1 & \mathrm{if}~ \omega>0 \\ 0 & \mathrm{if}~ \omega=0 \end{cases} \]
is called Gauss function (see Figure~\ref{fig:gaussMap}). This map is measure-preserving with regard to the measure $\mu$, which is defined by $\mu(A) = \frac{1}{\log 2}\int_A \frac{1}{1+x}~{d}x$ for all $A\in \A$ \cite{einsiedler2010}. The partition $\M=\{[\frac{1}{n+1},\frac{1}{n}[|~n\in\N\}\cup \{\{0\}\}$ of $[0,1]$ is a partition into monotony intervals of $T$. The map $T$ is countable piecewise monotone but not piecewise monotone. Therefore, we cannot use Theorem~\ref{thm:bandtPompe} to decide whether $\PE(T)$ and $h(T)$ are equal. However, we can use our new theorem below to show the equality as explained in Section~\ref{sec:discussion}.
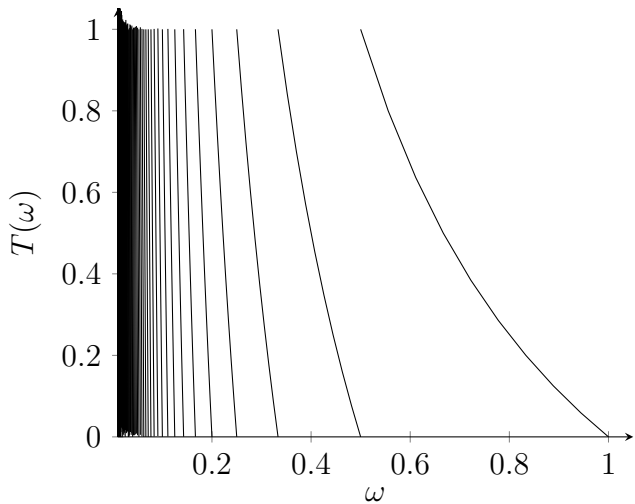
\begin{figure}[h]
\begin{tikzpicture}
\begin{axis}[
    axis lines = left,
    xlabel = $\omega$,
    ylabel = {$T(\omega)$},
    xmax = 1.05,
    ymin = 0,
    ymax = 1.05,
]
\foreach \i in {1,2,...,25}{
\pgfmathsetmacro{\l}{1/(\i+1)};
\pgfmathsetmacro{\r}{1/(\i)};
\addplot[
    domain=\l:\r, 
    samples=10, 
]
{1/x - \i};
}
\foreach \i in {26,...,100}{
\pgfmathsetmacro{\l}{1/(\i+1)};
\pgfmathsetmacro{\r}{1/(\i)};
\addplot[
    domain=\l:\r, 
    samples=3, 
]
{1/x - \i};
}
\end{axis}
\end{tikzpicture}
\caption{Graph of the Gauss function $T$.}
\label{fig:gaussMap}
\end{figure}
\label{ex:gaussMap}
\end{ex}

We will show here that \eqref{eq:PE=KS} is true for countable piecewise monotone maps $T$ as well. Our main result can be formulated as follows:

\begin{thm}[Main result]
Let $(\Omega,\B(\Omega),\mu,T)$ be a measure-preserving dynamical system with $\Omega\subseteq \R$ being a compact metric space. Let $T$ be aperiodic and countable piecewise monotone and $\M$ a countable partition into monotony intervals of $T$. If $H(\M)<\infty$ holds true, then
\[ \PE(T) = h(T). \]
\label{thm:mainThm}
\end{thm}

As already mentioned, the above theorem is a generalization of the result of Bandt, Pompe and G. Keller (Theorem~\ref{thm:bandtPompe}). Note that in the simpler case of piecewise monotony the restriction $H(\M)<\infty$ is not necessary because $H(\M)$ is always finite for a finite partition $\M$ into monotony intervals.\\
To prove our results, we begin with Lemma~\ref{lem:error} by reducing our problem to a combinatorial one. This is analogue to the approach used in \cite{bandt2002_2}. While Bandt et al. followed this by an examination of periodic points, utilizing the piecewise monotony and continuity, we use the piecewise monotony more directly and then apply measure theoretic arguments.\\
Bandt et al. showed the equality of the topological entropy and a topological variant of the permutation as well \cite{bandt2002_2}, which is generally not possible for maps with infinitely many monotony intervals, as Misiurewicz had shown \cite{misiurewicz2003}.

\section{Proofs}

\paragraph*{Rough outline of the proof}
Given a finite or countable partition $\P$, we show in subsection~\ref{sec:upperBound} that the entropy difference $\PE(T)-h(T,\P)$ is bounded from above by a term depending on the number of intersections between sets of points with an ordinal pattern and the sets of the partition $\Pn$. This number of intersections depends on the chosen partition $\P$.\\
Using the monotony of the map $T$, we chose the partition $\P$ as the partition into monotony intervals $\M$ in the next subsection. We show that for this choice of $\P$ the upper bound established in the previous subsection is finite. In particular, we show that the upper bound for the entropy difference depends on how frequently iterates of the map $T$ are lying within the same subset of the partition $\P$.\\
This fact allows us to create an arbitrary small upper bound for the entropy difference by constructing partitions $\P$ into monotony intervals in such a way that iterates of $T$ cannot stay within the same subset of $\P$ too frequently. This is done in the more technical subsection~\ref{sec:rokhlinTowers} by using Rokhlin towers.

\subsection{Upper bound for entropy difference}
\label{sec:upperBound}

We look at the difference between $\PE(T)$ and $h(T)$ by considering the refinement $\OP{n}\vee \Pn$ of an ordinal partition $\OP{n}$ and the partition $\Pn$ used in the definition of the Kolmogorov-Sinai entropy. The \textit{refinement} of two partitions $\P=\{P_i\}_{i\in I}$ and $\mathcal Q = \{Q_j\}_{j\in J}$ of $\Omega$ is defined by
\[ \P\vee \mathcal Q:=\{P_i\cap Q_j|~ i \in I, j \in J~\textrm{and}~P_i\cap Q_j \neq \emptyset \}.\]
One can easily verify that $H(\OP{n})\leq H(\Pn \vee \OP{n})$ is true. In the following lemma, we show that $H(\Pn \vee \OP{n})$ is bounded from above by $H(\Pn)$ plus some term depending on $n$ and the given partition $\P$. To find this term, we consider sets
\begin{equation}
S_n^{\P}(\bm i) := \{ \pi \in \Pi_n|~P(\bm i)\cap P_\pi \neq \emptyset \}
\label{eq:Sni}
\end{equation}
for $\bm i \in I^n$, which contain all permutations whose ordinal patterns are intersecting the set $P(\bm i)$. Roughly speaking, if the size of $S_n^{\P}(\bm i)$ does not grow too fast on average for increasing $n$, then the partition $\OP{n}$ does not add a lot of new information to $\Pn$, so $H(\Pn \vee \OP{n})$ and $H(\Pn)$ will be similar in size.\\
We give an upper bound on the difference between $\PE(T)$ and $h(T)$ based on $\card{S_n^{\P}(\bm i)}$, where $\card A$ denotes the number of elements in a set $A$:

\begin{lem}
Let $(\Omega,\B(\Omega),\mu,T)$ be a measure-preserving dynamical system with $\Omega\subseteq \R$ and $\P=\{P_i\}_{i\in I}$ a finite or countable partition of $\Omega$ with $H(\P)<\infty$. Then
\[ \PE(T) \leq h(T,\P)+ \liminf_{n\to\infty}\frac{1}{n} \sum_{\bm i\in I^n} \mu(P(\bm i))\log(\card{S_n^{\P}(\bm i)}) \]
holds true with $S_n^{\P}(\bm i)$ as defined in \eqref{eq:Sni}.
\label{lem:error}
\end{lem}

\begin{proof}
Let $\OP{n}$ be the partition into ordinal patterns of length $n$. Then
\begin{align*}
H(\OP{n}) & \leq H(\P^{(n)}\vee \OP{n})\\
&= -\sum_{\bm i\in I^n} \sum_{\pi \in \Pi_n} \mu(P(\bm i) \cap P_\pi)\log(\mu(P(\bm i)\cap P_\pi))\\
&= -\sum_{\bm i\in I^n} \sum_{\pi \in S_n^{\P}(\bm i)} \mu(P(\bm i) \cap P_\pi)\log(\mu(P(\bm i) \cap P_\pi))\\
&\leq -\sum_{\bm i\in I^n} \sum_{\pi \in S_n^{\P}(\bm i)} \frac{\mu(P(\bm i))}{\card{S_n^{\P}(\bm i)}}\log\left(\frac{\mu(P(\bm i))}{\card{S_n^{\P}(\bm i)}}\right)\\
&= -\sum_{\bm i\in I^n}  \mu(P(\bm i))\log\left(\frac{\mu(P(\bm i))}{\card{S_n^{\P}(\bm i)}}\right)\\
&= -\sum_{\bm i\in I^n} \mu(P(\bm i)) \log(\mu(P(\bm i))) + \sum_{\bm i\in I^n} \mu(P(\bm i)) \log(\card{S_n^{\P}(\bm i)})\\
&= H(\P^{(n)}) + \sum_{\bm i\in I^n} \mu(P(\bm i))\log(\card{S_n^{\P}(\bm i)}).
\end{align*}
Dividing both sides by $n$ and taking the limit inferior $n \to \infty$ finishes the proof.
\end{proof}

The above lemma is useful because it allows us to work with the number of elements of $S_n^\P(\bm i)$ and, therefore, to use combinatorial arguments. This is done in the following subsection.

\subsection{Using monotony}

Given a countable partition $\P=\{P_i\}_{i\in I}$ with finite entropy, the term
\[ \liminf_{n\to\infty}\frac{1}{n} \sum_{\bm i\in I^n} \mu(P(\bm i))\log(\card{S_n^{\P}(\bm i)}) \]
bounds the difference between $\PE(T)$ and $h(T,\P)$.\\
Our goal is now to find a sequence of countable partitions\\
$(\P_d)_{d\in \N} = ((P^d_i)_{i\in I_d})_{d\in \N}$ with finite entropy such that\\
$\liminf_{n\to\infty}\frac{1}{n} \sum_{\bm i\in I_d^n} \mu(P^d(\bm i))\log(\card{S_n^{\P_d}(\bm i)})$ converges to zero for $d\to\infty$. To achieve this, we construct the partitions $\P_d$ as special refinements of a given partition $\M$ into monotony intervals. This allows us to give an upper bound on the size of $\card{S_n^{\P_d}(\bm i)}$.

\begin{lem}
Let $T:\Omega\rightarrow \Omega$ be a countable piecewise monotone map on $\Omega \subseteq \R$ and $\M=\{M_i\}_{i\in I}$ a countable partition into monotony intervals of $T$. Then for all $n\in \N$ and multi indices $\bm{i}=(i_0,i_1,\ldots,i_{n-1}) \in I^n$
\[ \card{S_n^{\M}(\bm{i})}\leq2^{\card{ \{s\in \{0,1,...,n-2\}|~i_s = i_{n-1}\} }} \]
holds true.
\label{lem:Sn}
\end{lem}

In \eqref{eq:encodingPi}, an ordinal pattern of length $n$ was encoded by a permutation $\pi=(\pi_0,\pi_1,\ldots,\pi_{n-1}) \in \Pi_n$, where $T^{\pi_i}(\omega)$ is the $i$-th smallest element in the sequence $\omega,T(\omega),\ldots T^{n-1}(\omega)$.
To prove the above lemma, it helps to consider a different way of encoding ordinal patterns:\\
If we know for all $s,t\in \{0,1,\ldots,n-1\}$ whether $T^{ s}(\omega)\leq T^{ t}(\omega)$ is true, we can determine to which ordinal pattern $P_\pi$ the point $\omega$ belongs to. Therefore, we can encode an ordinal pattern by all pairwise comparisons of elements of the orbit of length $n$.

\begin{proof}[Proof of Lemma~\ref{lem:Sn}]
Fix a multi index $\bm{i}=(i_0,i_1,\ldots,i_{n-1}) \in I^n$ for $n\in \N$. Let
\[ E_n:=\{(s,t)\in\{0,1,...,n-1\}^2|~s<t\} \]
be the set of all possible pairs of ordered integers between $0$ and $n-1$.
Define the set
\[ \mathcal R_{s,t}:=\{\omega \in \Omega: T^s(\omega) \leq T^t(\omega)\}\]
for $(s,t)\in E_n$. We now use $0$ or $1$ to encode which sets $\mathcal R_{s,t}$ are intersecting the set $M(\bm i) = \bigcap_{k=0}^{n-1} T^{-k}(M_{i_k})$ by considering a set of specific functions on $E_n$ into $\{0,1\}$. Define
\[ F_n(\bm i):=\{f:E_n\rightarrow\{0,1\}|~M(\bm i) \cap \bigcap_{\substack{(s,t)\in E_n: \cr f(s,t)=1}} \mathcal R_{s,t}\cap \bigcap_{\substack{(s,t)\in E_n: \cr f(s,t)=0}} \mathcal R_{s,t}^c\neq \emptyset\}, \]
where $\mathcal R_{s,t}^c$ denotes the complement $\Omega\setminus \mathcal R_{s,t}$ of $\mathcal R_{s,t}$. Every permutation $\pi \in S_n^{\P}(\bm i)$ is uniquely determined by a function $f\in F_n(\bm i)$ and vice-versa. This implies
\[ \card{S_n^{\M}(\bm i)} = \card{F_n(\bm i)}. \]
Let
\[ E_n^d:=\{ (s,t)\in E_n|~t-s = d \} \]
be the set of pairs of integer in $E_n$ whose difference is equal to $d$ for $d\in \{1,\ldots,n-1\}$. Similarly, define
\begin{equation*}
F_n^d(\bm i) := \{f:E_n^d\rightarrow\{0,1\}|~M(\bm i) \cap \bigcap_{\substack{(s,t)\in E_n^d: \cr f(s,t)=1}} \mathcal R_{s,t}\cap \bigcap_{\substack{(s,t)\in E_n^d: \cr f(s,t)=0}} \mathcal R_{s,t}^c\neq \emptyset\}.
\end{equation*}
It is easy to see that $E_n= \bigcup_{d=1}^{n-1} E_n^d$ is true. Therefore, every function $f\in F_n$ is generated by a combination of functions $f_d\in F_n^d$ for $d\in\{1,\ldots n-1\}$ (but not every combinations of functions $f_d\in F_n^d$ necessarily generates a function $f\in F_n$). This implies
\begin{equation}
\card{F_n(\bm i)}\leq \prod_{d=1}^{n-1} \card{F_n^d(\bm i)}.
\label{eq:Fndi}
\end{equation}

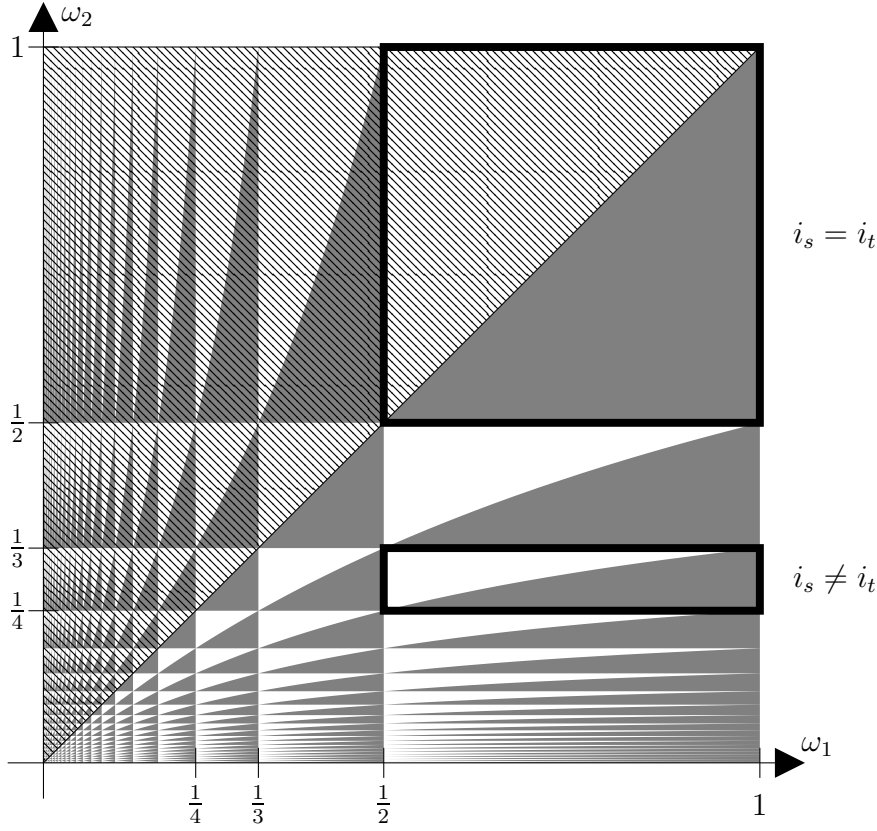
\begin{figure}[h]
\begin{tikzpicture}[scale = 10]
\pgfmathtruncatemacro{\nMax}{20};
\pgfmathsetmacro{\xStart}{1/(\nMax+1)};

\draw(-0.0,\xStart) -- (1.05,\xStart);
\draw(\xStart,-0.0) -- (\xStart,1.05);

\pgfmathsetmacro{\xl}{\xStart-0.02};
\pgfmathsetmacro{\xr}{\xStart+0.02};
\draw (1,\xl) -- (1,\xr);
\draw (\xl,1) -- (\xr,1);
\node[yshift = -10,xshift = -2] at (1,\xl) {{ $1$}};
\node[xshift = -6] at (\xl,1) {{ $1$}};

\foreach \i in {2,3,4}{
	\pgfmathsetmacro{\x}{1/\i}
	\draw (\x,\xl) -- (\x,\xr);
	\draw (\xl,\x) -- (\xr,\x);
	\node[yshift = -10,xshift = -2] at (\x,\xl) {{ $\frac{1}{\i}$}};
	\node[xshift = -6] at (\xl,\x) {{ $\frac{1}{\i}$}};
}

\fill (1.02,\xl) -- (1.06,\xStart) -- (1.02,\xr) -- cycle;
\node[yshift = 6] at (1.07,\xStart) {{ $\omega_1$}};

\fill (\xl,1.02) -- (\xStart,1.06) -- (\xr,1.02) -- cycle;
\node[xshift = 11] at (\xStart,1.04) {{ $\omega_2$}};

\foreach \i in {1,2,...,\nMax}{
	\foreach \j in {1,2,...,\nMax}{
		\pgfmathtruncatemacro{\kMax}{max(20-\i-\j,3)-1};
			\pgfmathsetmacro{\rFirst}{1/\i};
			\pgfmathsetmacro{\rSecond}{1/\j};
			\pgfmathsetmacro{\lFirst}{1/(\i+1)};
			\pgfmathsetmacro{\lSecond}{1/(\j+1)};
			\foreach \k in {1,...,\kMax}{
				\pgfmathsetmacro{\wOne}{\lFirst+(\rFirst-\lFirst)*(\k/(\kMax+1))};
				\pgfmathsetmacro{\wTwo}{1/(1/\wOne+\j-\i)};
			 	\coordinate (p-\k) at (\wOne,\wTwo);
			}
			\fill[opacity=0.5]
			  (\lFirst,\lSecond) \foreach \k in {1,...,\kMax} 	{--(p-\k)}--(\rFirst,\rSecond)--(\rFirst,\lSecond)--cycle;
	}
}

\draw[fill,pattern=north west lines] (\xStart,\xStart)--(\xStart,1)--(1,1)--cycle;

\draw[line width = 3] (1/2,1/3)--(1/2,1/4)--(1,1/4)--(1,1/3)--cycle;
\draw[line width = 3] (1,1/2)--(1/2,1/2)--(1/2,1)--(1,1)--cycle;

\node at (1.1,3/4) {$i_s = i_t$};
\node at (1.1,7/24) {$i_s \neq i_t$};
\end{tikzpicture}
\caption{The striped area corresponds to the set $R=\{(\omega_1,\omega_2)\in\Omega^2|~\omega_1\leq \omega_2\}$ and the gray area to $(T\times T)^{-1}(R)$ for the Gauss function $T$.}
\label{fig:R}
\end{figure}

Now fix $d\in \{1,\ldots n-1\}$. The problem of figuring out whether $f(s,t)$ can be $0$ or $1$ for $(s,t)\in E_n^d$ can be seen as the problem to deduce whether $(T^s(\omega),T^t(\omega))$ lies within $R:=\{(\omega_1,\omega_2)\in \Omega^2|~\omega_1\leq \omega_2\}$ or $R^c$ from the rectangle $M_{i_s}\times M_{i_t}$ in which $(T^s(\omega),T^t(\omega))$ is lying. The set $R$ corresponds to the striped triangle in Figure~\ref{fig:R}.\\
If $i_s \neq i_{t}$ holds true for $(s,d)\in E_n^d$, the points $T^{ s}(\omega)$ and $T^{ t}(\omega)$ lie in different intervals $M_{i_s}$ and $M_{i_t}$ for all $\omega \in M(\bm i)$. In Figure~\ref{fig:R}, this corresponds to the fact that $M_{i_s}\times M_{i_t}$ either completely lies in the triangle $R$ or in the triangle $R^c$. Therefore,
\begin{equation}
f(s,t)=0~\textrm{for all}~f\in F_n^d(\bm i)\quad\textrm{or}\quad f(s,t)=1~\textrm{for all}~f\in F_n^d(\bm i).
\label{eq:f(is=it)}
\end{equation}
If $i_s = i_{t}$ holds true for $(s,t)\in E_n^d$, the points $T^{ s}(\omega)$ and $T^{ t}(\omega)$ lie within the same interval $M_{i_s} = M_{i_t}$ for all $\omega \in M(\bm i)$. In Figure~\ref{fig:R} this corresponds to the fact that $M_{i_s}\times M_{i_t}$ is a square intersecting the diagonal of $\Omega^2$. So we cannot establish a straightforward equation like \eqref{eq:f(is=it)} that determines whether $f(s,t)$ is equal to $0$ or $1$.\\
Since $T$ acts monotonically on the interval $M_{i_s}$, applying the map $T$ on $T^{ s}(\omega)$ and $T^{ t}(\omega)$ preserves or reverses the order relation of $T^{ s}(\omega)$ and $T^{ t}(\omega)$, depending on whether $T$ is increasing or decreasing in $M_{i_s}$. Therefore,
\begin{equation*}
f(s,t)= \begin{cases}
f(s+1,t+1) & \mathrm{if}~T~\mathrm{is~ monotonically~ increasing~ in}~ M_{i_s},\\
1-f(s+1,t+1) & \mathrm{if}~T~\mathrm{is~ monotonically~ decreasing~ in}~ M_{i_s}
\end{cases}
\end{equation*}
for all $f\in F_n^d(\bm i)$ and $(s,t)\in E_n^d$ with $t<n-1$. In terms of Figure~\ref{fig:R}, this means that in each square $M_i\times M_i$ the set $R$ is equal to $(T\times T)^{-1}(R)$ if $T$ is monotonically increasing in $M_i$ and equal to $(T\times T)^{-1}(R^c)$ if $T$ is monotonically decreasing in $M_i$.\\
So every value $f(s,t)$ can be uniquely determined by the subsequent value $f(s+1,t+1)$ for all $t<n-1$, which implies
\[ \card{F_n^d((i_0,\ldots i_{n-1}))} = \card{F_{n-1}^d((i_1,\ldots i_{n-1}))} = \ldots = \card{F_{d+1}^d((i_{n-d-1},\ldots i_{n-1}))}. \]
Therefore, the value of $\card{F_n^d(\bm i)}$ does not depend on $n$ but on the amount of possibilities for the last value $f(n-1-d,n-1)$. Hence
\[ \card{F_n^d(\bm i)} = \card{ \{f(n-d-1,n-1)|~f\in F_n^d(\bm i)\}}. \]
Since $\card{ \{f(n-d-1,n-1)|~f\in F_n^d(\bm i)\}} = 1$ if $i_{n-d-1}\neq i_{n-1}$ according to \eqref{eq:f(is=it)} and there are at most $2$ possible outcomes for $f(n-d-1,n-1)$ otherwise, we have
\begin{equation*}
\card{F_n^d(\bm i)} \begin{cases}
= 1 & \mathrm{if}~i_{n-d-1} \neq i_{n-1},\\
\leq 2 & \mathrm{if}~ i_{n-d-1} = i_{n-1},
\end{cases}
\end{equation*}
for all $d\in\{1,2,\ldots,n-1\}$, which, together with~\eqref{eq:Fndi}, finishes the proof.
\end{proof}

As an immediate consequence of the above Lemma we get that the difference between Kolmogorov-Sinai entropy and permutation entropy is not larger than $\log 2$.

\begin{cor}
Let $(\Omega,\B(\Omega),\mu,T)$ be a measure-preserving dynamical system with $\Omega\subseteq \R$. Let $T$ be countable piecewise monotone and $\M$ a countable partition into monotony intervals of $T$. If $H(\M)<\infty$ holds true, then
\[ \PE(T)\leq h(T) + \log 2. \]
\label{cor:nonErgodic}
\end{cor}

\begin{proof}
We have
\[ \PE(T) \leq h(T,\M)+ \liminf_{n\to\infty}\frac{1}{n} \sum_{\bm i\in I^n} \mu(M(\bm i))\log(\card{S_n^{\M}(\bm i)}) \]
according to Lemma~\ref{lem:error}. Lemma~\ref{lem:Sn} provides
\[ \card{S_n^{\M}(\bm i)}\leq2^{\card{ \{s\in \{0,1,...,n-2\}|~i_s = i_{n-1}\} }} \leq 2^{n-1} \]
for all $n\in \N$. Combining the above statements yields
\begin{align*}
\PE(T) & \leq h(T,\M)+ \liminf_{n\to\infty}\frac{1}{n} \sum_{\bm i\in I^n} \mu(M(\bm i))(n-1)\log 2\\
&= h(T,\M)+ \liminf_{n\to\infty} (n-1)/n\log 2 = h(T,\M)+ \log 2\\
&\leq h(T) + \log 2. \qedhere
\end{align*}
\end{proof}

In order to give a smaller upper bound than $\log 2$, we have to put more effort into choosing our partition into monotony intervals $\M$.\\
Notice that a partition into monotony intervals of $T$ is not unique. For every countable partition $\M$ into monotony intervals and and every countable partition $\mathcal Q$ into intervals (but not necessarily into intervals of monotony) the partition $\M \vee \mathcal Q$ is a countable partition into monotony intervals as well.

%

\subsection{Rokhlin towers}
\label{sec:rokhlinTowers}

We want to construct partitions $\mathcal Q=\{Q_j\}_{j\in J}$ so that the expected value of $\card{ \{s\in \{0,1,...,n-2\}|~j_s = j_{n-1}\} }/n$ can be made small if $n$ goes to infinity. For $\card{ \{s\in \{0,1,...,n-2\}|~j_s = j_{n-1}\} }/n$ to be small, we try to construct the sets $Q_j$ in such a way that the iterates $T^{ s}(\omega)$ cannot stay in the same set $Q_j$ too frequently.\\
In the case of ergodic maps $T$, Birkhoff's ergodic theorem provides that the number of iterates $T^{s}(\omega)$ that are element of the set $Q_j$ is proportional to the measure of $Q_j$. So, by making the measure of $Q_j$ smaller for all $j\in J$ (which, consequently, increases the number of sets in $\mathcal Q$) we decrease the expected value of $\card{ \{s\in \{0,1,...,n-2\}|~j_s = j_{n-1}\} }/n$. Finally, one can use a similar approach to the one used in the proof of Theorem~\ref{thm:mainThm} to show the equality of permutation and KS-entropy. Since the ergodic case is contained in the general case, we do not show this here explicitly.\\
If $T$ is not ergodic, just making the measure of $Q\in \mathcal Q$ small is not enough any more, the sets $Q\in \mathcal Q$ need to be constructed in a specific  way. This is done in theorem~\ref{lem:nonErgodic} by using Rokhlin towers.

\begin{lem}[Rokhlin Lemma \cite{heinemann2000}]
Let $\Omega$ be a separable metric space,\\
$(\Omega,\B(\Omega),\mu,T)$ a measure-preserving dynamical system and $T$ aperiodic. Then for all $d\in \N$ and $\epsilon >0$ there exists a set $B\in\B(\Omega)$ with
\begin{enumerate}
\item $T^{-i}(B)\cap T^{-j}(B) = \emptyset$ for all $i,j\in\{0,...,d-1\}$ with $i\neq j$,
\item $\mu(B) \geq \frac{1-\epsilon}{d}$.
\end{enumerate}
The sequence of sets $(B,T^{-1}(B),\ldots,T^{-d+1}(B))$ is called Rokhlin tower of height $d$ with base $B$.
\label{lem:rokhlin}
\end{lem}

Such towers will turn out to be very useful because
\begin{equation}
\card{ \{s\in \{0,1,...,n-2\}|~T^{ s}(\omega) \in B\} } \leq \frac{n-2}{d}+1
\label{eq:RH-tower}
\end{equation}
is true for all $\omega\in\Omega$ if $B$ is a base of a Rokhlin tower of height $d$. We will show in the proof of Theorem~\ref{thm:mainThm} by using Lemma~\ref{lem:error} and \ref{lem:Sn} that, after dividing by $n$ and taking $n \to \infty$, inequality~\ref{eq:RH-tower} can be used to find an upper bound on the difference between Kolmogorov-Sinai and permutation entropy. Since we can construct Rokhlin tower of arbitrary height $d\in \N$, we can make this upper bound arbitrarily small by increasing $d$.\\
However, we cannot use $B,T^{-1}(B),\ldots,T^{-d+1}(B)$ directly to construct a countable partition $\mathcal Q$ of $\Omega$ into intervals because the sets\\ $B,T^{-1}(B),\ldots,T^{-d+1}(B)$ are generally not intervals. We need our partition $\mathcal Q$ to consist of intervals if we want to relate this partition to the ordinal partition $\OP{n}$ because two disjoint intervals are characterized by the fact that all elements in one interval are smaller than every element in the other interval. This fact does not need to be true any more if the sets $Q\in \mathcal Q$ are disconnected.\\
So the sets $B,T^{-1}(B),\ldots,T^{-d+1}(B)$ have to be approximated by sets of disjoint intervals, which is done with the help of the following two lemmas.

\begin{lem}[Approximation of Borel sets]
Let $(\Omega,\B(\Omega),\mu)$ be a probability space with $\Omega\subseteq \R$ and $\mu$ being a regular measure. Then for all $B\in \B(\Omega)$ and $\epsilon>0$, there exists $n\in\N$ and a finite number of pairwise disjoint intervals $A_i, i\in\{1,\ldots,n\}$, such that
\begin{equation*}
\mu\left(\left(\bigcup_{i=1}^n A_i\right)\triangle B \right) \leq \epsilon
\end{equation*}
holds true, where $\mu(A\triangle B) = \mu(A\setminus B)+\mu(B\setminus A)$ is the symmetric difference between sets $A$ and $B$.
\label{lem:approx}
\end{lem}

\begin{proof}
Take $B\in \B(\Omega)$ and $\epsilon>0$. Because $\mu$ is regular, we can find an open set $O \in \B(\Omega)$ with $O \supseteq B$ and
\[ \mu(O\triangle B) = \mu(O \setminus B) \leq \frac{\epsilon}{2}.\]
Since $O$ is open, there exists a countable collection of pairwise disjoint open intervals $A_i$, $i\in \N$, with $O = \bigcup_{i\in \N} A_i$. Choose $n\in \N$ so that
\[ \mu\left(\left(\bigcup_{i=1}^n A_i\right)\triangle O \right) = \mu\left(\bigcup_{i=n+1}^{\infty}A_i\right) \leq \frac{\epsilon}{2} \]
is true. This implies with the triangle inequality
\begin{align*}
&\mu\left(\left(\bigcup_{i=1}^n A_i\right)\triangle B \right) \leq \mu\left(\left(\bigcup_{i=1}^n A_i\right)\triangle O \right) + \mu\left(O \triangle B \right)\\
& \leq \mu\left(\bigcup_{i=n+1}^\infty A_i \right) + \frac{\epsilon}{2} \leq \epsilon.\qedhere
\end{align*}
\end{proof}

So we can approximate the base $B$ of a Rokhlin tower by a set of disjoint interval $A_i, i=1,\ldots,n$ with arbitrary precision. Since $T$ is measure-preserving, we can also approximate $T^{-k}(B)$ by the sets $T^{-k}(A_i)$ but these sets do not need to be intervals any more. However, one can use the piecewise monotony of $T$ to show that $T^{-k}(A_i)$ can be written as a finite or countable union of intervals:\\

%

For a given (countable) piecewise monotone map $T:\Omega\rightarrow\Omega$, one can easily see that $T^{-1}(A)$ is a finite (or countable) union of intervals for every interval $A\subseteq \Omega$. In fact, if $A$ is an interval and $\M$ the partition into monotony intervals of $T$, the set $T^{-1}(A) \cap M$ is an interval for every $M  \in \M$. This implies that for every partition $\P$ into intervals the partition $\P':=T^{-1}(\P)\vee \M$ will be a partition into intervals. Analogously,
\[ T^{-1}(\P') \vee \M = T^{-2}(\P) \vee T^{-1}(M) \vee \M \]
will be a partition into intervals as well. Repeating this argument provides that
\[ T^{-k}(\P) \vee \left(\bigvee_{l=0}^{k-1}T^{-l}(\M)\right) \]
is a partition into intervals for all $k\in \N$. Since the intersection of two intervals is an interval again,
\begin{equation}
\left(\bigvee_{k=1}^{d-1}T^{-k}(\P)\right)\vee \left(\bigvee_{k=0}^{d-2}T^{-k}(\M)\right)
\label{eq:intervalPartition}
\end{equation}
is a partition into intervals for all $d\in \N$.\\

Applying Lemma~\ref{lem:approx} and using the piecewise monotony of $T$ as explained above provides us with many sets of disjoint intervals. In the proof of the following lemma, we combine those intervals into a countable partition $\P=\{P_i\}_{i\in I}$ of disjoint intervals. We are then interested in specific elements $A \in \sigma(\P)$, where $\sigma(\P)$ is the smallest $\sigma$-algebra containing all sets $P\in\P$. Since $\P$, as well as the index set $I$, is countable and all sets $P\in \P$ are pairwise disjoint, we can explicitly state $\sigma(\P)$ as
\[ \sigma(\P)= \left\{\bigcup_{j\in J}P_j|~J\subseteq I\right\}. \]
In particular, this implies that every set $A\in\sigma(\P)$ can be expressed as
\[ A = \bigcup_{i \in I_A}P_i \]
with a unique countable index set $I_A\subseteq I$. We denote by
\[ split(A|\P) := \{P_i\}_{i\in I_A} \]
the countable collection of sets in $\P$ into which $A$ is split.

%
%

\begin{thm}
Let $(\Omega,\B(\Omega),\mu,T)$ be a measure-preserving dynamical system with $\Omega\subseteq \R$ being a compact metric space. Let $T$ be aperiodic, countable piecewise monotone and $\M$ a countable partition into monotony intervals of $T$. Then for all $\epsilon>0$ and $d\in \N$ there exists a partition $\mathcal Q=\{Q_j\}_{j\in J}$ of $\Omega$ and an index set $\widetilde J\subseteq J$, such that
\begin{enumerate}[label=(\roman*)]
\item $\mathcal Q$ consist of countably many intervals,
\label{lem:nonErgodic:1}
\item $Q_j\cap T^{-k}(Q_j)=\emptyset$ for all $k\in\{1,2,\ldots d-1\}$ and $j \in \widetilde J$,
\label{lem:nonErgodic:4}
\item $H(\M)<\infty$ implies $H(\mathcal Q)<\infty$,
\label{lem:nonErgodic:5}
\item and $\mu(\bigcup_{j \in \widetilde J} Q_j)\geq1-\epsilon$.
\label{lem:nonErgodic:3}
\end{enumerate}
\label{lem:nonErgodic}
\end{thm}

\begin{proof}
Take $\epsilon>0$ and $d\in \N$. According to Lemma~\ref{lem:rokhlin}, there exists a set $B\in \B(\Omega)$ with
\begin{equation}
B \cap T^{-k}(B) = \emptyset
\label{eq:BEmpty}
\end{equation}
for all $k\in\{1,2,\ldots,d-1\}$ and $\mu(B)\geq\frac{1-\epsilon/2}{d}$. Since any Borel probability measure on a compact metric space is regular \cite{parthasarathy1967}, we can apply Lemma~\ref{lem:approx} to $B$, which provides the existence of a finite number of disjoint intervals $A_i, i\in\{1,\ldots,n\}$ such that
\begin{equation}
\mu\left(\left(\bigcup_{i=1}^n A_i\right)\triangle B \right) \leq \frac{\epsilon}{6d^2}.
\label{eq:epsilonH}
\end{equation}
Consider $\widetilde \Omega:= [\inf \Omega, \sup \Omega]$ and $\widetilde R:=\widetilde \Omega\setminus\bigcup_{i=1}^nA_i$. Because $\widetilde \Omega$ and $A_i$ are intervals for all $i\in \{1,\ldots,n\}$, there exists $m\in \N$ and intervals $\widetilde R_i\subseteq \widetilde \Omega$ with
\[ \widetilde R= \bigcup _{i=1}^{m} \widetilde R_i. \]
Define $R_i:= \widetilde R_i \hspace{0.025cm}\cap \hspace{0.025cm}\Omega$ for all $i\in \{1,\ldots,m\}$. Since $M\in \M$ and $\widetilde R_i$ are intervals,
\[ M\cap R_i = M \cap (\widetilde R_i \cap \Omega) = (M \cap \Omega) \cap \widetilde R_i =  M \cap \widetilde R_i \]
is an interval for all $i \in \{1,\ldots,m\}$ and $M\in \M$. So $\{R_i\}_{i=1}^m \vee \M$ is a countable collection of intervals and $\{A_i\}_{i=1}^n \cup \{R_i\}_{i=1}^m \vee \M$ a countable partition of $\Omega$ into intervals.\\
Consider the partition
\begin{equation*}
\P:=\left(\bigvee_{l=0}^{d-1}T^{-l} \left(\bigvee_{k=0}^{d-1}T^{-k}\left(\{A_i\}_{i=1}^n \cup \{R_i\}_{i=1}^m \vee \M \right)  \right)\right)\vee \left( \bigvee_{k=0}^{2d-3}T^{-k}(\M) \right).
\end{equation*}
It follows from \eqref{eq:intervalPartition} that $\P$ is a countable partition into intervals.\\
Now define
\[ \widehat A_i := A_i\setminus \left(\bigcup_{k=1}^{d-1}\bigcup_{j=1}^{n} T^{-k}(A_j) \right) \]
for $i\in \{1,2,\ldots,n\}$.\\
For all $u,v\in \{1,2,\ldots,n\}$ and $k,l\in \{0,1,\ldots,d-1\}$ with $k < l$ we have
\begin{align*}
& T^{-k}(\widehat A_u) \cap T^{-l}(\widehat A_v) = T^{-k}(\widehat A_u \cap T^{-(l-k)}(\widehat A_v))\\
\subseteq& T^{-k}\left( \left(A_u\setminus \left(\bigcup_{s=0}^{d-1} \bigcup_{j=1}^n T^{-s}(A_j) \right)\right) \cap T^{-(l-k)}(A_v)\right)\\
\subseteq& T^{-k}\left( T^{-(l-k)}(A_v)\setminus \left(\bigcup_{s=0}^{d-1} \bigcup_{j=1}^n T^{-s}(A_j) \right) \right) = \emptyset.
\end{align*}
For $k = l$ but $u \neq v$
\[ T^{-k}(\widehat A_u) \cap T^{-k}(\widehat A_v) = T^{-k}(\widehat A_u \cap \widehat A_v) \subseteq T^{-k}(A_u \cap A_v) = \emptyset\]
holds true. So
\begin{equation}
T^{-k}(\widehat A_u) \cap T^{-l}(\widehat A_v) = \emptyset
\label{eq:hatA_iDisjoint}
\end{equation}
is fulfilled for all $u,v\in \{1,2,\ldots,n\}$ and $k,l\in \{0,1,\ldots,d-1\}$ with $u \neq v$ or $k\neq l$.\\
We have
\begin{equation*}
T^{-l}(\widehat A_i) \in \sigma(\P)
\end{equation*} for all $i\in \{1,2,\ldots n\}$ and $l\in \{0,1,\ldots,d-1\}$. This implies
\[ \widehat R:= \Omega \setminus \left(\bigcup_{l=0}^{d-1}  \bigcup_{i=1}^n T^{-l}(\widehat A_i)\right) \in \sigma(\P). \]
Since $\P$ is a countable partition into intervals, $T^{-l}(\widehat A_i)$ and $\widehat R$ can be expressed as the union of countably many intervals $P\in\P$. The partition $\mathcal Q$ that consists of those intervals is defined as
\begin{equation*}
\mathcal Q:=~ \{Q_j\}_{j\in J}:=~ split\left(\widehat R|\P\right)~\cup~\bigcup_{l=0}^{d-1}  \bigcup_{i=1}^n~ split\left(T^{-l}(\widehat A_i)|\P\right).
\end{equation*}
The collection of sets $\mathcal Q$ is indeed a partition of $\Omega$ because\\
$\{\widehat R\} \cup \bigcup_{l=0}^{d-1} \bigcup_{i=1}^n \{T^{-l}(\widehat A_i)\}$ is a partition of $\Omega$. Notice that, because $\P$ is a countable partition into intervals and $d$ is finite, $\mathcal Q$ is a countable partition into intervals as well. So \ref{lem:nonErgodic:1} is fulfilled.\\
Choose
\[ \widetilde J:=\{j\in J|~Q_j\subseteq \bigcup_{l=0}^{d-1} \bigcup_{i=1}^n T^{-l}(\widehat A_i) \} \]
and take $j \in J_0$. Then there exists a set $\widehat A_i$ and $l\in \{0,1,\ldots,d-1\}$ with $Q_j\subseteq T^{-l}(\widehat A_i)$. So for each $k\in \{1,\ldots,d-1\}$ we have
\begin{equation*}
Q_j\cap T^{-k}(Q_j) \subseteq T^{-l}(\widehat A_i)\cap T^{-k}(T^{-l}(\widehat A_i)) = T^{-l}(\widehat A_i \cap T^{-k}(\widehat A_i)) = \emptyset.
\end{equation*}
Therefore, \ref{lem:nonErgodic:4} is true.\\
Because $\sigma(\mathcal Q)\subseteq \sigma(\P)$, we have
\begin{align*}
&H(\mathcal Q) \leq H(\P)\\
&= H\left(\left(\bigvee_{l=0}^{d-1}T^{-l} \left(\bigvee_{k=0}^{d-1}T^{-k}\left(\{A_i\}_{i=1}^n \cup \{R_i\}_{i=1}^m \vee \M\right)  \right)\right)\vee \left( \bigvee_{k=0}^{2d-3}T^{-k}(\M) \right)\right)\\
&\leq H\left(\bigvee_{l=0}^{d-1}T^{-l} \left(\bigvee_{k=0}^{d-1}T^{-k}\left(\{A_i\}_{i=1}^n \cup \{R_i\}_{i=1}^m \vee \M\right)  \right)\right) + H \left( \bigvee_{k=0}^{2d-3}T^{-k}(\M) \right)\\
&\leq \sum_{l=0}^{d-1}H\left(T^{-l} \left(\bigvee_{k=0}^{d-1}T^{-k}\left(\{A_i\}_{i=1}^n \cup \{R_i\}_{i=1}^m \vee \M\right)  \right)\right) + \sum_{k=0}^{2d-3}H \left(T^{-k}(\M) \right)\\
&= d\cdot H\left(\bigvee_{k=0}^{d-1}T^{-k}\left(\{A_i\}_{i=1}^n \cup \{R_i\}_{i=1}^m \vee \M\right)  \right) + 2(d-1)\cdot H \left(\M \right)\\
&\leq d\cdot \sum_{k=0}^{d-1}H\left(T^{-k}\left(\{A_i\}_{i=1}^n \cup \{R_i\}_{i=1}^m \vee \M\right)  \right) + 2(d-1)\cdot H \left(\M \right)\\
&= d^2\cdot H\left(\{A_i\}_{i=1}^n \cup \{R_i\}_{i=1}^m \vee \M\right) + 2(d-1)\cdot H \left(\M \right)\\
&\leq d^2\cdot H\left(\left(\{A_i\}_{i=1}^n \cup \{R_i\}_{i=1}^m\right) \vee \M\right) + 2(d-1)\cdot H \left(\M \right)\\
&\leq d^2\cdot H\left(\{A_i\}_{i=1}^n \cup \{R_i\}_{i=1}^m\right) + (2(d-1)+d^2)\cdot H \left(\M \right)
\end{align*}
So \ref{lem:nonErgodic:5} holds true. It remains to show \ref{lem:nonErgodic:3}:\\
Using $\widehat A_i\subseteq A_i$ and the fact that the sets $A_i$ are pairwise disjoint provides
\begin{align*}
&\mu\left(\left(\bigcup_{i=1}^n A_i\right) \setminus \left(\bigcup_{i=1}^n \widehat A_i\right) \right) = \sum_{i=1}^n \mu(A_i\setminus \widehat A_i)\\
=& \sum_{i=1}^n \mu\left(A_i\cap \left(\bigcup_{k=1}^{d-1}\bigcup_{j=1}^{n} T^{-k}(A_j)\right)\right) = \sum_{i=1}^n \mu\left(\bigcup_{k=1}^{d-1} \bigcup_{j=1}^{n}(A_i\cap T^{-k}(A_j))\right)\\
=& \mu\left(\bigcup_{i=1}^n\bigcup_{k=1}^{d-1} \bigcup_{j=1}^{n}(A_i\cap T^{-k}(A_j))\right) \leq \sum_{k=1}^{d-1} \mu\left(\bigcup_{i=1}^n \bigcup_{j=1}^{n}(A_i\cap T^{-k}(A_j))\right)\\
\displaybreak
=& \sum_{k=1}^{d-1} \mu\left(\bigcup_{i=1}^n \left( A_i \cap T^{-k}\left(\bigcup_{j=1}^{n} A_j\right)\right)\right) = \sum_{k=1}^{d-1} \mu\left(\left(\bigcup_{i=1}^n A_i \right) \cap T^{-k}\left(\bigcup_{j=1}^{n} A_j\right)\right)\\
=& \sum_{k=1}^{d-1} \left[ \mu\left(\left(\left(\bigcup_{i=1}^n A_i \right)\cap B\right) \cap T^{-k}\left(\left(\bigcup_{j=1}^{n} A_j\right)\cap B\right)\right) \right. \\
&+ \mu\left(\left(\left(\bigcup_{i=1}^n A_i \right)\cap B\right) \cap T^{-k}\left(\left(\bigcup_{j=1}^{n} A_j\right)\setminus B\right)\right)\\
&+ \left. \mu\left(\left(\left(\bigcup_{i=1}^n A_i \right)\setminus B\right) \cap T^{-k}\left(\bigcup_{j=1}^n A_j\right)\right) \right]\\
\stackrel{\eqref{eq:BEmpty}}{\leq}& \sum_{k=1}^{d-1} \left[ \mu\left(T^{-k}\left(\left(\bigcup_{j=1}^{n} A_j\right)\setminus B \right)\right) + \mu\left(\left(\bigcup_{i=1}^{n} A_i\right)\setminus B \right) \right]\\
=& \sum_{k=1}^{d-1} 2\cdot\mu\left(\left(\bigcup_{i=1}^{n} A_i\right)\setminus B \right) 
= 2(d-1)\cdot\mu\left(\left(\bigcup_{i=1}^nA_i\right)\setminus B \right)\\
\leq& 2(d-1)\cdot\mu\left(\left(\bigcup_{i=1}^nA_i\right)\triangle B \right) \stackrel{\eqref{eq:epsilonH}}{\leq} \frac{2(d-1)\cdot\epsilon}{6 d^2} \leq \frac{\epsilon}{3d}.
\end{align*}
This implies
\begin{align*}
&\mu(B) \leq \mu\left(\bigcup_{i=1}^n A_i\right) + \mu\left(\left(\bigcup_{i=1}^n A_i\right)\triangle B\right) \stackrel{\eqref{eq:epsilonH}}{\leq} \mu\left(\bigcup_{i=1}^n A_i\right) + \frac{\epsilon}{6d^2}\\
&\leq \mu\left(\bigcup_{i=1}^n\widehat A_i\right) +\mu\left(\left(\bigcup_{i=1}^n A_i\right) \setminus \left(\bigcup_{i=1}^n \widehat A_i\right) \right) + \frac{\epsilon}{6d^2}\\
&\leq \mu\left(\bigcup_{i=1}^n\widehat A_i\right) +\frac{\epsilon}{3d} + \frac{\epsilon}{6d^2} \leq \mu\left(\bigcup_{i=1}^n\widehat A_i\right) +\frac{\epsilon}{3d} + \frac{\epsilon}{6d} = \mu\left(\bigcup_{i=1}^n\widehat A_i\right) +\frac{\epsilon}{2d},
\end{align*}
which is equivalent to
\begin{equation*}
\mu\left(\bigcup_{i=1}^n\widehat A_i\right) \geq \mu(B)-\frac{\epsilon}{2d} \geq \frac{1-\epsilon/2}{d} - \frac{\epsilon/2}{d} = \frac{1-\epsilon}{d}.
\end{equation*}
Hence
\begin{align*}
&\mu\left(\bigcup_{j \in \widetilde J} Q_j \right) = \mu\left(\bigcup_{l=0}^{d-1} \bigcup_{i=1}^n T^{-l}(\widehat A_i)\right) \stackrel{\eqref{eq:hatA_iDisjoint}}{=} \sum_{l=0}^{d-1} \mu\left(T^{-l}\left(\bigcup_{i=1}^n \widehat A_i\right)\right)\\
&= \sum_{l=0}^{d-1} \mu\left(\bigcup_{i=1}^n \widehat A_i\right) = d\cdot\mu\left(\bigcup_{i=1}^n \widehat A_i\right) \geq d\cdot \frac{1-\epsilon}{d} = 1-\epsilon.
\end{align*}
So \ref{lem:nonErgodic:3} is fulfilled.
\end{proof}

We now combine a partition $\mathcal Q$ as described in Lemma~\ref{lem:nonErgodic} and a partition $\M$ into monotony intervals of $T$ into the partition $\P = \M \vee \mathcal Q$. This partition combines the properties of $\mathcal Q$ and $\M$, i.e. we can apply Lemma~\ref{lem:Sn} to $\P$ as we could to $\M$ and the properties given in Lemma~\ref{lem:nonErgodic} are true for $\P$ as they were for $\mathcal Q$.

\begin{proof}[Proof of Theorem~\ref{thm:mainThm}]
The inequality $\PE(T)\geq h(T)$ follows from \eqref{eq:inequality}. We now have to show $\PE(T)\leq h(T)$.\\
Let $\M = \{M_i\}_{i \in I}$ be a finite or countable partition into monotony intervals of $T$ with $H(\M)<\infty$. For any $d\in \N$ choose a countable partition
\[ \mathcal Q_{d}: = \{Q^d_j\}_{j\in J_d} \]
of $\Omega$ into intervals and an index set $\widetilde J_d\subseteq J_d$ with
\begin{itemize}
\item $Q_j^d\cap T^{-k}(Q_j^d)=\emptyset$ for all $k\in\{1,2,\ldots d-1\}$ and $j \in \widetilde J_d$,
\item $H(\mathcal Q)<\infty$,
\item and $\mu(\bigcup_{j \in \widetilde J_d} Q_j^d)\geq1-\frac{1}{d}$.
\end{itemize}
According to Theorem~\ref{lem:nonErgodic}, this is always possible. Consider the partition
\[ \P_{d} := \M \vee \mathcal Q_{d} = \{M_i\cap Q_j^d\}_{(i,j)\in I\times J_d}.\]
Applying Lemma~\ref{lem:error} to $\P_{d}$ yields
\begin{equation}
\PE(T) \leq h(T,\P_{d})+ \liminf_{n\to\infty}\frac{1}{n} \sum_{(\bm i,\bm j)\in (I\times J_d)^n} \mu(P^{d}((\bm i,\bm j)))\log(\card{S_n^{\P_{d}}((\bm i,\bm j))}),
\label{eq:2combi1}
\end{equation}
where we consider $(\bm i,\bm j)$ itself as one multi index and $I\times J_d$ as one index set. Note that $\P_{d}$ is a countable partition into monotony intervals of $T$ for all $d\in\N$. Therefore, we can apply Lemma~\ref{lem:Sn} to $\card{S_n^{\P_{d}}((\bm i,\bm j))}$, which yields
\begin{align*}
&\sum_{(\bm i,\bm j)\in (I\times J_d)^n} \mu(P^{d}((\bm i,\bm j)))\log(\card{S_n^{\P_{d}}((\bm i,\bm j))})\\
\leq \log 2 &\sum_{(\bm i,\bm j)\in (I\times J_d)^n} \mu(P^{d}((\bm i,\bm j)))\card{ \{s\in \{0,1,...,n-2\}|~(i_s,j_s) = (i_{n-1},j_{n-1})\} }\\
\leq \log 2 &\sum_{(\bm i,\bm j)\in (I\times J_d)^n} \mu(P^{d}((\bm i,\bm j)))\card{ \{s\in \{0,1,...,n-2\}|~j_s = j_{n-1}\} }\\
= \log 2 &\sum_{\bm i \in I^n}\sum_{\bm j \in J_d^n} \mu(P^{d}((\bm i,\bm j)))\card{ \{s\in \{0,1,...,n-2\}|~j_s = j_{n-1}\} }\\
= \log 2 &\sum_{\bm j \in J_d^n} \mu(Q^{d}(\bm j))\card{ \{s\in \{0,1,...,n-2\}|~j_s = j_{n-1}\} }. \numberthis
\label{eq:2combi2}
\end{align*}
We have $H(\P_d) \leq H(\M) + H(\mathcal Q_d)<\infty$ for all $d\in\N$ and, consequently,
\begin{equation}
\liminf_{d\to \infty} h(T,\P_{d}) \leq h(T).
\label{eq:2combi3}
\end{equation}
Combining \eqref{eq:2combi1},\eqref{eq:2combi2} and \eqref{eq:2combi3} yields
\begin{equation*}
 \PE(T) \leq h(T) + \log 2\liminf_{d\to\infty}\liminf_{n\to\infty}\frac{1}{n} \sum_{\bm j \in J_d^n} \mu(Q^{d}(\bm j))\card{ \{s\in \{0,1,...,n-2\}|~j_s = j_{n-1}\} }.
\end{equation*}
So it remains to show that
\begin{equation}
\liminf_{d\to\infty}\liminf_{n\to\infty}\frac{1}{n} \sum_{\bm j \in J_d^n} \mu(Q^{d}(\bm j))\card{ \{s\in \{0,1,...,n-2\}|~j_s = j_{n-1}\}} = 0
\label{eq:2remains}
\end{equation}
is true.\\
We have
\begin{align*}
&\sum_{\bm j \in J_d^n} \mu(Q^{d}(\bm j))\card{ \{s\in \{0,1,...,n-2\}|~j_s = j_{n-1}\}}\\
=& \sum_{j_{n-1}\in J_d}\sum_{\bm j \in J_d^{n-1}} \mu(Q^{d}((\bm j,j_{n-1})))\card{ \{s\in \{0,1,...,n-2\}|~j_s = j_{n-1}\} }\\
=& \sum_{j_{n-1}\in \widetilde J_d}\sum_{\bm j \in J_d^{n-1}} \mu(Q^{d}((\bm j,j_{n-1})))\card{ \{s\in \{0,1,...,n-2\}|~j_s = j_{n-1}\} }\\
&+ \sum_{j_{n-1}\in J_d\setminus \widetilde J_d}\sum_{\bm j \in J_d^{n-1}} \mu(Q^{d}((\bm j,j_{n-1})))\card{ \{s\in \{0,1,...,n-2\}|~j_s = j_{n-1}\} }\\
\stackrel{\eqref{eq:RH-tower}}{\leq}& \sum_{j_{n-1}\in \widetilde J_d}\sum_{\bm j \in J_d^{n-1}} \mu(Q^{d}((\bm j,j_{n-1}))) \left(\frac{n-2}{d}+1\right)\\
&+ \sum_{j_{n-1}\in J_d\setminus \widetilde J_d}\sum_{\bm j \in J_d^{n-1}} \mu(Q^{d}((\bm j,j_{n-1})))(n-1)\\
\leq& \frac{n-2}{d}+1 + \sum_{j_{n-1}\in J_d\setminus \widetilde J_d} \mu( Q^d_{j_{n-1}})(n-1) \leq \frac{n-2}{d} + \frac{n-1}{d}+1 = \frac{2n-3}{d}+1
\end{align*}
for any $n,d\in \N$. This implies equality~\eqref{eq:2remains} and finishes the proof.
\end{proof}

\section{Discussion}
\label{sec:discussion}

\begin{rem}
The requirement of $\mu$ being aperiodic is not a significant restriction of the main result:\\
Consider the set of periodic points
\[ Per:= \bigcup_{t=1}^\infty\{\omega\in \Omega|~\omega=T^t(\omega) \}. \]
We can divide the measure $\mu$ into a periodic part $\mu_p$ with $\mu_p(A) := \frac{\mu(A\cap Per)}{\mu(Per)}$ and an aperiodic part $\mu_a$ with $\mu_a(A):=\frac{\mu(A\setminus Per)}{1-\mu(Per)}$ for all $A\in \A$. Since $Per$ is $\mu$-almost surely $T$-invariant, i.e. $\mu(Per\triangle T^{-1}(Per)) = 0$, both $\mu_p$ and $\mu_a$ are $T$-invariant measures. Therefore, $\mu=\mu(Per)\mu_p + (1-\mu(Per))\mu_a$ implies (see e.g. \cite{walters2000})
\begin{equation}
h(T,\mu) = \mu(Per)h(T,\mu_p) + (1-\mu(Per))h(T,\mu_a)
\label{eq:hAffin}
\end{equation}
By using the same arguments as in \cite{walters2000} for the proof of \eqref{eq:hAffin}, one can verify that
\begin{equation}
\PE(T,\mu) = \mu(Per)\PE(T,\mu_p) + (1-\mu(Per))\PE(T,\mu_a)
\label{eq:hPEAffin}
\end{equation}
holds true, where $h(T,\mu)$ and $\PE(T,\mu)$ denote the corresponding entropies of the dynamical system $(\Omega,\B(\Omega),\mu,T)$.\\
We can apply our main theorem to $h(T,\mu_a)$ and $\PE(T,\mu_a)$ because $T$ is aperiodic with regards to $\mu_a$. Since the dynamics of periodic points are determined by a finite number of iterations, on can show that $h(T,\mu_p) = \PE(T,\mu_p) = 0$ is true. Combining \eqref{eq:hAffin} and \eqref{eq:hPEAffin} yields
\[ \PE(T,\mu) = (1-\mu(Per)) \PE(T,\mu_a) = (1-\mu(Per)) h(T,\mu_a) = h(T,\mu). \]
\end{rem}

\paragraph{Continuing example~\ref{ex:gaussMap}}
In order to apply Theorem~\ref{thm:mainThm} to the Gauss function $T$, we have to check if $H(\mathcal M)<\infty$ is true for $\M=\{[\frac{1}{n+1},\frac{1}{n}[|~n\in\N\} \cup \{\{0\}\}$:\\
Consider the map $\phi:]0,\infty]\rightarrow ]0,\infty]$ with $\phi(x)=-x\log(x)$ for all $x>0$. The map $\phi$ is monotonically increasing on $]0,1/e[$. Choose $N\in \N$ so that
\[ \frac{1}{n}-\frac{1}{n+1} = \frac{1}{n(n+1)}<\frac{\log 2}{e} \]
is true for all $n\geq N$. Recall that the invariant measure $\mu$ of the Gauss function $T$ is defined by $\mu([a,b[) = \frac{1}{\log 2}\int_a^b \frac{1}{1+x}~{d}x$ for $0\leq a<b\leq 1$. We have
\[ \mu\left(\left[\frac{1}{n+1},\frac{1}{n}\right[\right) = \frac{1}{\log 2}\int_{\frac{1}{n+1}}^{\frac{1}{n}} \frac{1}{1+x}~{d}x \leq \frac{1}{\log 2}\int_{\frac{1}{n+1}}^{\frac{1}{n}} 1~{d}x = \frac{1}{\log 2} \cdot\frac{1}{n(n+1)} \]
for all $n\in \N$. This implies
\begin{align*}
&\phi\left(\mu\left(\left[\frac{1}{n+1},\frac{1}{n}\right[\right)\right) \leq \phi\left(\frac{1}{\log 2} \cdot\frac{1}{n(n+1)}\right)\\
&\leq \phi\left(\frac{1}{\log 2} \cdot\frac{1}{n^2}\right) = \frac{\log \log 2 + 2\log n}{(\log 2) n^2} \numberthis
\label{eq:ineqPhi}
\end{align*}
for all $n\geq N+1$. So we can conclude
\begin{align*}
H(\M) &= \sum_{n=1}^\infty \phi\left(\mu\left(\left[\frac{1}{n+1},\frac{1}{n}\right[ \right)\right)\\
&= \sum_{n=1}^{N} \phi\left(\mu\left(\left[\frac{1}{n+1},\frac{1}{n}\right[ \right)\right) + \sum_{n=N+1}^\infty \phi\left(\mu\left(\left[\frac{1}{n+1},\frac{1}{n}\right[ \right)\right)\\
&\stackrel{\eqref{eq:ineqPhi}}{\leq} \sum_{n=1}^{N} \phi\left(\mu\left(\left[\frac{1}{n+1},\frac{1}{n}\right[ \right)\right) + \sum_{n=N+1}^\infty \frac{\log \log 2 + 2\log n}{(\log 2) n^2}\\
&\leq \sum_{n=1}^{N} \phi\left(\mu\left(\left[\frac{1}{n+1},\frac{1}{n}\right[ \right)\right) + \sum_{n=1}^\infty \frac{\log \log 2}{(\log 2) n^2} + \sum_{n=1}^\infty \frac{2\log n}{(\log 2) n^2}< \infty,
\end{align*}
which allows us to apply Theorem~\ref{thm:mainThm} and get $\PE(T)= h(T)$.

\paragraph{Generalizing interval notion}
The main reason why the partition $\M=\{M_i\}_{i\in I}$ into sets of monotony for the map $T$ was required to be a partition into intervals is the fact that a collection of disjoint intervals can be ordered in such a way that the order relation between two different intervals corresponds to the order relation of the points within those intervals. This information about the order relation was utilized in \eqref{eq:f(is=it)} as one part of determining the number of sets with ordinal patterns of length $n$ intersecting a set $\M(\bm i)$ for $\bm i\in I^n$. To describe this specific ordering on the set of those intervals, we write
\begin{equation}
M_i<M_j ~ \mathrm{if}~ \omega_i<\omega_j~\mathrm{holds~ true~ for~ all}~ (\omega_i,\omega_j)\in M_i\times M_j.
\label{eq:orderedM}
\end{equation}
We can generalize this ordering of intervals in a way that allows us to ignore sets of points with measure zero and that preserves the correspondence between the order of the different sets $M_i$ and points within those sets. To achieve this, we write $A<_\mu B$ for $A,B\in \B(\Omega)$, if
\[ (\mu\times \mu)(\{ (\omega_1,\omega_2)\in A\times B|~\omega_1\geq \omega_2 \})=0 \]
holds true, where $(\mu\times \mu)$ is the product measure on the $\sigma$-algebra $\B(\Omega^2)$. So $A<_\mu B$ means that almost every point in $A$ is smaller than almost every point in $B$, which can be interpreted as a probabilistic formulation of \eqref{eq:orderedM}.\\
We say that $\M = \{M_i\}_{i\in I}$ is an ordered partition of $\Omega$, if $M_i<_\mu M_j$ or $M_j<_\mu M_i$ is true for all $i,j\in I$ with $i\neq j$.
Since elements of an ordered partition can be ordered the same way basic intervals could, up to sets with measure zero, our main theorem remains valid if we consider such a generalized partition as a partition into monotony sets.

\paragraph*{Acknowledgement}
The authors would like to thank Mike Todd (University of St. Andrews) for his valuable suggestions.

\printbibliography

\end{document}